\documentclass{article}
\usepackage[T1]{fontenc}
\usepackage[utf8]{inputenc}
\usepackage[style=alphabetic,sorting=nyt,giveninits=true]{biblatex}
\usepackage{geometry}
\usepackage{parskip}
\usepackage{amsmath}
\usepackage{amssymb}
\usepackage{amsthm}
\usepackage{hyperref}
\usepackage[capitalize,nameinlink,noabbrev]{cleveref}
\usepackage{mathtools}
\usepackage{commath}
\usepackage{bm}
\usepackage{tikz}
\usetikzlibrary{cd,fit,shapes.geometric}
\usepackage{adjustbox}
\usepackage{array}
\usepackage{enumitem}
\usepackage{mathrsfs}
\usepackage{pgfplots}
\pgfplotsset{compat=1.17}
\usepackage{tcolorbox}
\usepackage{xparse}
\usepackage{microtype}
\usepackage{multirow}
\usepackage{tabularray}
\usepackage{hhline}

\tikzset{
	smallPosetDiagram/.style={
    	x=1.8cm,
    	y=1.8cm,
    	baseline={([yshift=-0.5ex]current bounding box.center)}
	}
}

\tikzset{
	largePosetDiagram/.style={
    	x=1.6cm,
    	y=1.4cm
	}
}

\tikzset{
	smallGraph/.style={
    	x=0.9cm,
		y=0.9cm,
		every node/.style={circle,draw=black,fill=black,inner sep=0pt,minimum size=5pt},
		label distance=0.145cm,
		line width=0.25mm
	}
}

\tikzset{
	largeGraph/.style={
	  x=0.75cm,
	  y=0.75cm,
	  baseline={([yshift=0.5ex]current bounding box.center)}
	}
}

\ExplSyntaxOn
  \cs_new:Npn \__get_max_length:nN #1#2 {
    \dim_set:Nn \l_tmpa_dim { 0pt }
    \clist_map_inline:nn { #1 } {
      \hbox_set:Nn \l_tmpa_box {##1}
      \dim_compare:nT {\l_tmpa_dim < \box_wd:N \l_tmpa_box}
      {
        \dim_set:Nn \l_tmpa_dim { \box_wd:N \l_tmpa_box }
      } 
    }
    \dim_set_eq:NN #2 \l_tmpa_dim
  }

  \NewDocumentCommand{\maxlength}{}{\__get_max_length:nN}
\ExplSyntaxOff

\makeatletter
\newcommand{\myitem}[1]{%
\item[#1]\protected@edef\@currentlabel{#1}%
}
\makeatother

\makeatletter
\def\thmCite{\@ifnextchar[{\@with}{\@without}}
\def\@with[#1]#2{{\normalfont\cite[#1]{#2}\;}}
\def\@without#1{{\normalfont\cite{#1}\;}}
\makeatother

\makeatletter
\newcommand*{\defeq}{\hspace{-0.04cm}\mathrel{\rlap{%
                     \raisebox{0.4ex}{$\scriptstyle\m@th\cdot$}}%
                     \raisebox{-0.05ex}{$\scriptstyle\m@th\cdot$}}%
				 =}
\makeatother

\ExplSyntaxOn
\DeclareExpandableDocumentCommand{\IfNoValueOrEmptyTF}{mmm}
 {
  \IfNoValueTF{#1}{#2}
   {
    \tl_if_empty:nTF {#1} {#2} {#3}
   }
 }
\ExplSyntaxOff

\newcommand*{\valOrBlank}[1]{\IfNoValueOrEmptyTF{#1}{}{#1}}

\NewDocumentCommand{\myOplus}{e{_^}}{%
	\hspace{0.025cm}\bigoplus_{\mathclap{\valOrBlank{#1}}}^{\mathclap{\valOrBlank{#2}}}\hspace{0.025cm}%
}

\NewDocumentCommand{\mySum}{e{_^}}{%
	\sum_{\mathclap{\valOrBlank{#1}}}^{\mathclap{\valOrBlank{#2}}}\hspace{0.025cm}%
}

\NewDocumentCommand{\myCap}{e{_^}}{%
	\bigcap_{\mathclap{\valOrBlank{#1}}}^{\mathclap{\valOrBlank{#2}}}\hspace{0.05cm}%
}

\NewDocumentCommand{\myCup}{e{_^}}{%
	\bigcup_{\mathclap{\valOrBlank{#1}}}^{\mathclap{\valOrBlank{#2}}}\,%
}

\tcbset{sharp corners, colback=white, colframe=black, boxrule=0.4pt,parbox=false}

\tikzset{square/.style={regular polygon,regular polygon sides=4}}

\begingroup
\makeatletter
\@for\theoremstyle:=definition,remark,plain\do{%
	\expandafter\g@addto@macro\csname th@\theoremstyle\endcsname{%
		\addtolength\thm@preskip\parskip
	}%
}
\endgroup

\let\oldproof\proof
\def\proof{\oldproof\unskip}

\makeatletter
\def\footnoterule{
  \hrule \@width 9cm \kern 3\p@} 
\makeatother

\DeclareMathOperator{\ara}{ara}

\DeclareMathOperator{\cd}{cd}

\DeclareMathOperator{\depth}{depth}

\DeclareMathOperator{\lastGrade}{end}
\DeclareMathOperator{\reg}{reg}

\DeclarePairedDelimiter{\smallAbs}{|}{|}

\newtheorem{theorem}{Theorem}[section]
\newtheorem{lemma}[theorem]{Lemma}
\newtheorem{corollary}[theorem]{Corollary}
\newtheorem{proposition}[theorem]{Proposition}
\newtheorem{definition}[theorem]{Definition}

\newtheorem{notation}[theorem]{Notation}
\theoremstyle{definition}
\newtheorem{example}[theorem]{Example}
\theoremstyle{remark}
\newtheorem*{note}{Note}

\newcommand{\blockComment}[1]{}

\DeclareTextFontCommand{\emph}{\boldmath\bfseries}

\makeatletter
\g@addto@macro\bfseries{\boldmath}
\makeatother

\AtBeginBibliography{\vspace*{8pt}}

\DeclareFieldFormat[article,book,incollection,misc]{title}{\mkbibitalic{#1}}
\DeclareFieldFormat[article]{journaltitle}{#1\isdot}
\DeclareFieldFormat[article,incollection]{booktitle}{\normalfont{#1}}

\title{Binomial Edge Ideals of Complements of Graphs\\ of Girth at Least 5}
\author{David Williams\footnote{The author was supported by EPSRC grant EP/T517835/1}}
\date{\vspace{-0.3cm}}

\makeatletter
\renewcommand*{\@fnsymbol}[1]{\@arabic{#1}}
\makeatother

\begin{document}

\maketitle

\begin{abstract}
	\noindent We calculate the local cohomology modules of the binomial edge ideals of the complements of graphs of girth at least $5$ using the tools introduced by \`{A}lvarez Montaner in \cite[Section 3]{alvarezmontanerLocalCohomologyBinomial2020}.  We then use this calculation to compute the depth, dimension, and regularity of these binomial edge ideals.
\end{abstract}

\begin{tcolorbox}
	Unless specified otherwise, let
	\begin{equation*}
		R=k[x_1,\ldots,x_n,y_1,\ldots,y_n]
	\end{equation*}
	for some field $k$ (of arbitrary characteristic) and $n\geq 2$. We denote by $\mathfrak{m}$ the irrelevant ideal of $R$. Furthermore, set $\delta_{i,j}\defeq x_iy_j-x_jy_i$ for $1\leq i<j\leq n$.

	Throughout this paper, all graphs will be finite, simple and undirected. For a graph $G$, we denote by $V(G)$ and $E(G)$ the sets of vertices and edges $\{i,j\}$ of $G$ respectively. Unless specified otherwise, all graphs will have at most $n$ vertices, with vertex set $\{1,\ldots,\smallAbs{V(G)}\}$.

	We denote by $P_m$ the path on $m$ vertices, $K_m$ the complete graph on $m$ vertices, and $K_S$ the complete graph with vertex set $S$.

	Finally, for a graph $G$ and $m\geq1$, we denote by $mG$ the graph given by the disjoint union of $m$ copies of $G$.
\end{tcolorbox}

\section{Preliminaries}

\subsection{Binomial Edge Ideals}

Our objects of study are the binomial edge ideals of graphs. These  were introduced in \cite{herzogBinomialEdgeIdeals2010}, and independently in \cite{ohtaniGraphsIdealsGenerated2011}.

We begin with their definition (most of this subsection appeared previously in \cite[Section 1]{williamsLFCoversBinomialEdge2023}):

\begin{definition}\label{BEIDef}
	Let $G$ be a graph. We define the \emph{binomial edge ideal} of $G$ as
	\begin{equation*}
		\mathcal{J}(G)\defeq(\delta_{i,j}:\{i,j\}\in E(G))
	\end{equation*}
\end{definition}

There is a rich interplay between the algebraic properties of these ideals and the combinatorial properties of the corresponding graph. For an overview of some of these interactions, see \cite[Chapter 7]{herzogBinomialIdeals2018}.

A key fact about binomial edge ideals is the following:

\begin{theorem}{\normalfont\cite[Corollary 2.2]{herzogBinomialEdgeIdeals2010}}
	Binomial edge ideals are radical.
\end{theorem}

In fact, an explicit description of the primary decomposition can be given by purely combinatorial means, but we must first introduce some notation:

\begin{notation}
	Let $G$ be a graph, and $S\subseteq V(G)$. Denote by $c_G(S)$ the number of connected components of $G\setminus S$ (we will just write $c(S)$ when there is no confusion). Then we set
	\begin{equation*}
		\mathcal{C}(G)\defeq\{S\subseteq V(G):\text{\normalfont$S=\varnothing$ or $c(S\setminus\{v\})<c(S)$ for all $v\in S$}\}
	\end{equation*}
	That is, when we ``add back'' any vertex in $S\in\mathcal{C}(G)$, it must reconnect some components of $G\setminus S$.
\end{notation}

\begin{notation}
	Let $G$ be a graph, and $S\subseteq V(G)$. Denote by $G_1,\ldots,G_{c(S)}$ the connected components of $G\setminus S$, and let $\tilde{G}_i$ be the complete graph with vertex set $V(G_i)$. Then we set
	\begin{equation*}
		P_S(G)\defeq(x_i,y_i:i\in S)+\mathcal{J}(\tilde{G}_1)+\cdots+\mathcal{J}(\tilde{G}_{c(S)})
	\end{equation*}
\end{notation}

\begin{proposition}\label{CompleteGraphPrime}
	$\mathcal{J}(K_S)$ is prime for any $S\subseteq\{1,\ldots,n\}$, and therefore $P_T(G)$ is prime for any graph $G$ and $T\subseteq V(G)$ also.
\end{proposition}

\begin{proof}
	This follows from \cite[Theorem 2.10]{brunsDeterminantalRings1988}.
\end{proof}

\begin{theorem}{\normalfont\cite[Corollary 3.9]{herzogBinomialEdgeIdeals2010}}\label{BEIPD}
	Let $G$ be a graph. Then
	\begin{equation*}
		\mathcal{J}(G)=\myCap_{S\in\mathcal{C}(G)}P_S(G)
	\end{equation*}
	is the primary decomposition of $\mathcal{J}(G)$.
\end{theorem}

We will make use of the following properties of $\mathcal{J}(K_n)$:

\begin{proposition}\label{CompleteGraphDim}
	We have that $R/\mathcal{J}(K_n)$ is Cohen-Macaulay of dimension $n+1$.
\end{proposition}

\begin{proof}
	This follows from \cite[Example 1.7~(a) \& Corollary 3.4]{herzogBinomialEdgeIdeals2010}.
\end{proof}

\begin{proposition}\label{CompleteGraphReg}
	$\reg_R(R/\mathcal{J}(K_n))=1$, where $\reg$ denotes the (Castelnuovo-Mumford) regularity.
\end{proposition}

\begin{proof}
	We have
	\begin{equation*}
		\reg_R(\mathcal{J}(K_n))=2
	\end{equation*}
	by \cite[Theorem 2.1]{kianiBinomialEdgeIdeals2012} (see \cite[Remark 3.3]{kianiBinomialEdgeIdeals2012}), and so the result is immediate from the fact that
	\begin{equation*}
		\reg_R(R/I)=\reg_R(I)-1
	\end{equation*}
	for any (homogeneous) ideal $I$ of $R$.
\end{proof}

The properties of path graphs are also well understood:
\begin{proposition}\label{PathGraphDim}
	We have that $R/\mathcal{J}(P_n)$ is a complete intersection ring of dimension $n+1$.
\end{proposition}

\begin{proof}
	This follows from \cite[Corollary 1.2]{eneCohenMacaulayBinomialEdge2011}.
\end{proof}

\begin{note}
	Gonz\'{a}lez-Martin\'{e}z shows in \cite[Theorem A]{gonzalez-martinezGorensteinBinomialEdge2021} that, for a connected graph $G$, $R/\mathcal{J}(G)$ is Gorenstein if and only if $G$ is a path.
\end{note}

\begin{proposition}\label{PathGraphReg}
	$\reg_R(R/\mathcal{J}(P_n))=3$.
\end{proposition}

\begin{proof}
	This follows from \cite[Theorem 1.1]{matsudaRegularityBoundsBinomial2013} and, again, the fact that
	\begin{equation*}
		\reg_R(R/I)=\reg_R(I)-1
	\end{equation*}
	for any (homogeneous) ideal $I$ of $R$.
\end{proof}

\pagebreak

\subsection{\`{A}lvarez Montaner's Hochster-Type Formula for Local Cohomology Modules of Binomial Edge Ideals}

The key tool we will make use of is a Hochster-type decomposition of the local cohomology modules of binomial edge ideals given by the reduced cohomology of intervals in the order complex of a certain poset associated to the binomial edge ideal, which was introduced by \`{A}lvarez Montaner in \cite[Section 3]{alvarezmontanerLocalCohomologyBinomial2020}.

We begin by defining the poset itself:

\begin{definition}\label{AMPosetDef}
	For a graph $G$, we construct the poset $\mathcal{Q}_{\mathcal{J}(G)}$ as follows:
	\begin{enumerate}
		\item We start with the associated primes $\mathfrak{p}_1,\ldots,\mathfrak{p}_t$ of $\mathcal{J}(G)$.
		\item Next, we add all sums of these primes to our poset.
		\item If any of these sums is not prime, replace it with the primes in its primary decomposition.
		\item Now add to the poset all sums of the previous elements and these new primes.
		\item Repeat this process until all ideals in the poset are prime and all sums are included.
	\end{enumerate}
	This process terminates after a finite number of steps, as explained in \cite[Definition 3.3]{alvarezmontanerLocalCohomologyBinomial2020}.

	We then order these ideals by reverse inclusion, and adjoin a maximal element $1_{\mathcal{Q}_{\mathcal{J}(G)}}$.
\end{definition}

\begin{note}
	These ideals are all Cohen-Macaulay, as explained in the introduction of \cite[Section 3.2]{alvarezmontanerLocalCohomologyBinomial2020}.
\end{note}

\begin{example}
	If
	\begin{equation*}
		G=\quad\begin{tikzpicture}[largeGraph]
			\foreach \v in {1,...,4}
				\node (\v) at (\v,1) {};
			\foreach \i in {1,...,3}
				\pgfmathsetmacro\j{int(\i+1)}
				\draw[line width=0.3mm] (\i.center) -- (\j.center);
			\foreach \v in {1,...,4}
				\fill[black] (\v) circle (2.75pt) node [below=0.1cm] {$\v$};
		\end{tikzpicture}
	\end{equation*}
	then $1_{\mathcal{Q}_{\mathcal{J}(G)}}$ is given by
	\begin{center}	
		\begin{tikzpicture}[x=4cm,y=1.75cm]
			\node (1) at (2,6) {$1_{\mathcal{Q}_{\mathcal{J}(G)}}$};
			\node (J) at (1,5) {$P_\varnothing(G)$};
			\node (a2) at (2,5) {$P_{\{2\}}(G)$};
			\node (b2) at (1,4) {$P_{\varnothing}(G)+P_{\{2\}}(G)$};
			\node (a3) at (3,5) {$P_{\{3\}}(G)$};
			\node (b3) at (2,4) {$P_{\varnothing}(G)+P_{\{3\}}(G)$};
			\node (c2-3) at (3,4) {$P_{\{2\}}(G)+P_{\{3\}}(G)$};
			\node (d2-3) at (2,3) {$P_{\varnothing}(G)+P_{\{2\}}(G)+P_{\{3\}}(G)$};

			\foreach \from/\to in {J/1,a2/1,b2/J,b2/a2,a3/1,b3/J,b3/a3,c2-3/a2,c2-3/a3,d2-3/b2,d2-3/b3,d2-3/c2-3}
				\draw[-to] (\from) -- (\to);
		\end{tikzpicture}
	\end{center}
\end{example}

We now describe a simplicial complex associated to a poset:

\begin{definition}
	Let $(\mathcal{P},\leq)$ be a poset. Then we define the \emph{order complex} associated to $(\mathcal{P},\leq)$ as the simplicial complex whose facets are the maximal chains in $\mathcal{P}$.

	Given $S_1,S_2\in\mathcal{P}$ with $S_1\leq S_2$, we denote by $(S_1,S_2)$ the order complex on that open interval in the poset, that is, the complex with facets given by the maximal chains strictly between $S_1$ and $S_2$.
\end{definition}

\pagebreak

The key theorem is then as follows:

\begin{theorem}\thmCite[Theorem 3.9]{alvarezmontanerLocalCohomologyBinomial2020}\label{AMMainTheorem}
	We have isomorphisms
	\begin{equation*}
		H_\mathfrak{m}^r(R/\mathcal{J}(G))\cong\bigoplus_{\mathclap{\mathfrak{q}\in\mathcal{Q}_{\mathcal{J}(G)}}}H_\mathfrak{m}^{d_\mathfrak{q}}(R/\mathfrak{q})^{M_{r,\mathfrak{q}}}
	\end{equation*}
	of graded $k$-vector spaces for all $r\geq0$, where $d_\mathfrak{q}=\dim(R/\mathfrak{q})$ and
	\begin{equation*}
		M_{r,\mathfrak{q}}=\dim_k(\widetilde{H}^{r-d_\mathfrak{q}-1}((\mathfrak{q},1_{\mathcal{Q}_{\mathcal{J}(G)}});k))\label{MrqDef}
	\end{equation*}
\end{theorem}

\begin{example}\label{P4Example}
	Again, let
	\begin{equation*}
		G=\quad\begin{tikzpicture}[largeGraph]
			\foreach \v in {1,...,4}
				\node (\v) at (\v,1) {};
			\foreach \i in {1,...,3}
				\pgfmathsetmacro\j{int(\i+1)}
				\draw[line width=0.3mm] (\i.center) -- (\j.center);
			\foreach \v in {1,...,4}
				\fill[black] (\v) circle (2.75pt) node [below=0.1cm] {$\v$};
		\end{tikzpicture}
	\end{equation*}
	and set
	\begin{center}
		$\begin{aligned}[t]
			\mathfrak{j}&\defeq P_\varnothing(G)\\
			\mathfrak{a}_v&\defeq P_{\{v\}}(G)\\
			\mathfrak{b}_v&\defeq P_{\{v\}}(G)+P_\varnothing(G)\\
			\mathfrak{c}_{\{2,3\}}&\defeq P_{\{2\}}(G)+P_{\{3\}}(G)\\
			\mathfrak{d}_{\{2,3\}}&\defeq P_\varnothing(G)+P_{\{2\}}(G)+P_{\{3\}}(G)\\
		\end{aligned}
		\hspace{0.5cm}
		\begin{aligned}[t]
			&\\
			&\text{\normalfont for $v\in\{2,3\}$}\\
			&\text{\normalfont for $v\in\{2,3\}$}\\
			&\\
			&
		\end{aligned}$
	\end{center}
	(we will generalise this notation later in the paper).

	Then we have
	\begin{center}
		\begin{tblr}{colspec={c|c|c|c|c|c},hline{2} = {1}{-}{solid},hline{2} = {2}{-}{solid},vline{2-6} = {abovepos = 1, belowpos = 1},stretch=1.75}
			$\mathfrak{q}$ & $\dim_R(R/\mathfrak{q})$ & $(\mathfrak{q},1_{\mathcal{Q}_{\mathcal{J}(\overline{G})}})$ & $\dim_k(\widetilde{H}^{-1})$ & $\dim_k(\widetilde{H}^0)$ & $\dim_k(\widetilde{H}^1)$ \\
			$\mathfrak{j}$ & $5$ & $\varnothing$ & $1$ & $0$ & $0$ \\
			\hline
			$\mathfrak{a}_v$ & $5$ & $\varnothing$ & $1$ & $0$ & $0$ \\
			\hline
			$\mathfrak{b}_v$ & $4$ &
			\begin{tikzpicture}[smallPosetDiagram]
				\node (j) at (0,1.035) {$\mathfrak{j}$};
				\node (av) at (1,1) {$\mathfrak{a}_v$};
			\end{tikzpicture}
			& $0$ & $1$ & $0$ \\
			\hline
			$\mathfrak{c}_{\{2,3\}}$ & $4$ &
			\begin{tikzpicture}[smallPosetDiagram]
				\node (a2) at (0,1) {$\mathfrak{a}_2$};
				\node (a3) at (1,1) {$\mathfrak{a}_3$};
			\end{tikzpicture}
			& $0$ & $1$ & $0$ \\
			\hline
			$\mathfrak{d}_{\{2,3\}}$ & $3$ &
			\begin{tikzpicture}[smallPosetDiagram]
				\node (j) at (0,1) {$\mathfrak{j}$};
				\node (a2) at (1,1) {$\mathfrak{a}_2$};
				\node (a3) at (2,1) {$\mathfrak{a}_3$};
				\node (b2) at (0,0) {$\mathfrak{b}_2$};
				\node (b3) at (1,0) {$\mathfrak{b}_3$};
				\node (c2-3) at (2,0) {$\mathfrak{c}_{\{2,3\}}$};

				\foreach \from/\to in {b2/j,b2/a2,b3/j,b3/a3,c2-3/a2,c2-3/a3}
					\draw[-to] (\from) -- (\to);
			\end{tikzpicture}
			& $0$ & $0$ & $1$
		\end{tblr}
	\end{center}
	and so
	\begin{equation*}
		H_\mathfrak{m}^5(R/\mathcal{J}(G))\cong H_\mathfrak{m}^5(R/\mathfrak{j})\oplus\hspace{-0.0625cm}\left[\hspace{0.3cm}\bigoplus_{\mathclap{v\in\{2,3\}}}\,H_\mathfrak{m}^5(R/\mathfrak{a}_v)\right]\hspace{-0.1cm}\oplus\hspace{-0.1cm}\left[\hspace{0.3cm}\bigoplus_{\mathclap{v\in\{2,3\}}}\,H_\mathfrak{m}^4(R/\mathfrak{b}_v)\right]\hspace{-0.1cm}\oplus H_\mathfrak{m}^4(R/\mathfrak{c}_{\{2,3\}})\oplus H_\mathfrak{m}^3(R/\mathfrak{d}_{\{2,3\}})
	\end{equation*}
	as graded $k$-vector spaces, with all other local cohomology modules vanishing.

	This tells us that $R/\mathcal{J}(G)$ is Cohen-Macaulay of dimension $5$.
\end{example}

\pagebreak

\section{Complements of Graphs of Girth at Least 5}

\begin{definition}
	Let $G$ be a graph. We define the \emph{girth} of $G$ to be the number of vertices in the smallest cycle contained in $G$, or $\infty$ if $G$ is acyclic (that is, if $G$ is a forest).
\end{definition}

We will first deal with some trivial situations separately to avoid including lots of additional cases in later theorems:

\subsection{Some Trivial Cases}\label{TrivialCases}

\subsubsection{\texorpdfstring{The Case When $G$ Has No Edges}{The Case When G Has No Edges}}

When $G$ consists of $n$ isolated vertices (that is, $G=nK_1$), we have $\overline{G}=K_n$, and so this case is addressed by \cref{CompleteGraphDim} and \cref{CompleteGraphReg}.

\subsubsection{\texorpdfstring{The Case When $G$ Has a Universal Vertex}{The Case When G Has a Universal Vertex}}

If $G$ has a universal vertex (that is, a vertex adjacent to all others) and girth at least $5$, then $G$ must be a star, since any other edges would form a $3$-cycle. In this case, it is straightforward to calculate the local cohomology of $R/\mathcal{J}(\overline{G})$, and therefore its dimension, depth, and regularity:

\begin{proposition}
	Suppose that $G$ is a star on $n$ vertices. Then
	\begin{enumerate}
		\item $R/\mathcal{J}(\overline{G})$ is Cohen-Macaulay of dimension $n+2$.
		\item $\reg_R(R/\mathcal{J}(\overline{G}))=1$.
	\end{enumerate}
\end{proposition}

\begin{proof}
	Note that $\overline{G}=K_{n-1}\sqcup K_1$. The first claim is then immediate from \cref{CompleteGraphDim}, and the second from \cref{CompleteGraphReg} and \cref{FreeVarReg}.
\end{proof}

\subsubsection{\texorpdfstring{The Case When $n\leq3$}{The Case When n≤3}}
It is easily checked that the only graphs of girth at least $5$ on at most $3$ vertices which are not stars are $2K_1$, $3K_1$, and $K_2\sqcup K_1$. These have complements $K_2$, $K_3$ and $P_3$ respectively, all of which are Cohen-Macaulay of dimension $n+1$ by \cref{CompleteGraphDim} and \cref{PathGraphDim}.

We know that $\reg_R(R/\mathcal{J}(K_n))=1$ by \cref{CompleteGraphReg}, and $\reg_R(R/\mathcal{J}(P_3))=3$ by \cref{PathGraphReg}.

\subsection{Some Properties of Complements of Graphs of Girth at Least 5}

We begin with some definitions and notation:

\begin{definition}
	For a graph $G$ and vertex $v\in V(G)$, we denote by $N_G(v)$ the \emph{neighbourhood} of $v$ in $G$, that is, the vertices $w$ in $G$ for which $\{v,w\}\in E(G)$ (recall that our edges are undirected). Furthermore, we denote by $N_G[v]$ the \emph{closed neighbourhood} of $v$ in $G$, which is given by $N_G(v)\cup\{v\}$.
\end{definition}

\begin{definition}\label{FreeEdgeDef}
	If $\{v,w\}\in E(G)$ is its own connected component, we say that $\{v,w\}$ is a \emph{free edge} of $G$. We denote by $F(G)$ the subset of $E(G)$ consisting of all free edges of $G$.
\end{definition}

\begin{note}
	\cref{FreeEdgeDef} is not standard notation.
\end{note}

\begin{tcolorbox}[label=GAssumptions]
	For the remainder of this paper, unless specified otherwise, $G$ will denote a graph of girth at least $5$ with \emph{no universal vertex, at least one edge, and $n\geq4$}, since we have \hyperref[TrivialCases]{already addressed} the other cases.

	We denote by $H$ the subgraph of $G$ obtained by removing all leaves (that is, vertices of degree $1$) and isolated vertices of $G$.

	For clarity, we will denote by $\overline{v}$ the vertex in $\overline{G}$ corresponding to the vertex $v$ in $G$, although the two are really the same.
\end{tcolorbox}

\pagebreak

We first describe $\mathcal{C}(\overline{G})$:

\begin{lemma}\label{ComplementCutSets}
	We have
	\begin{equation*}
		\mathcal{C}(\overline{G})=\{\varnothing\}\cup\{N_{\overline{G}}(\overline{v}):v\in V(H)\}\cup\{V(G)\setminus\{v,w\}:\{v,w\}\in F(G)\}
	\end{equation*}
\end{lemma}

\begin{proof}
	Trivially we have $\varnothing\in\mathcal{C}(\overline{G})$. Take any other $S\in\mathcal{C}(\overline{G})$. We will first show that $\overline{G}\setminus S$ consists of exactly two connected components. We know that it must have more than one. Say we had the induced subgraph
	
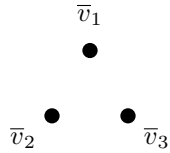
\begin{equation*}
		\begin{tikzpicture}[smallGraph]
			\node[label={-150:$\overline{v}_2$}](v2) at (0,0) {};
			\node[label={90:$\overline{v}_1$}](v1) at ([shift=(60:1)]v2) {};
			\node[label={-30:$\overline{v}_3$}](v3) at ([shift=(0:1)]v2) {};
		\end{tikzpicture}
	\end{equation*}
	in $\overline{G}\setminus S$. Then we would have the $3$-cycle
	
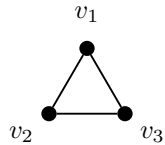
\begin{equation*}
		\begin{tikzpicture}[smallGraph]
			\node[label={-150:$v_2$}](v2) at (0,0) {};
			\node[label={90:$v_1$}](v1) at ([shift=(60:1)]v2) {};
			\node[label={-30:$v_3$}](v3) at ([shift=(0:1)]v2) {};

			\foreach \from/\to in {v1/v2,v1/v3,v2/v3}
				\draw[-] (\from) -- (\to);
		\end{tikzpicture}
	\end{equation*}
	in $G$, a contradiction.

	Next we will show that at least one of these components must be an isolated vertex. Suppose that we had the induced subgraph
	
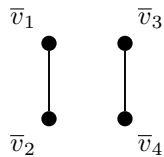
\begin{equation*}
  		\begin{tikzpicture}[smallGraph]
     		\node[label={135:$\overline{v}_1$}](v1) at (0,1) {};
      		\node[label={-135:$\overline{v}_2$}](v2) at (0,0) {};
     		\node[label={45:$\overline{v}_3$}](v3) at (1,1) {};
      		\node[label={-45:$\overline{v}_4$}](v4) at (1,0) {};

			\foreach \from/\to in {v1/v2,v3/v4}
				\draw[-] (\from) -- (\to);
    	\end{tikzpicture}
	\end{equation*}
	in $\overline{G}\setminus S$. Then we would have the $4$-cycle
	
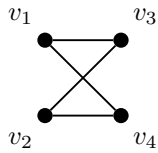
\begin{equation*}
  		\begin{tikzpicture}[smallGraph]
     		\node[label={135:$v_1$}](v1) at (0,1) {};
      		\node[label={-135:$v_2$}](v2) at (0,0) {};
     		\node[label={45:$v_3$}](v3) at (1,1) {};
      		\node[label={-45:$v_4$}](v4) at (1,0) {};

			\foreach \from/\to in {v1/v3,v1/v4,v2/v3,v2/v4}
				\draw[-] (\from) -- (\to);
    	\end{tikzpicture}
	\end{equation*}
	in $G$, again a contradiction, and so at least one component must be an isolated vertex as claimed.

	Let $\overline{v}$ denote an isolated vertex in $\overline{G}\setminus S$. We first suppose that the other connected component of $\overline{G}\setminus S$ is also an isolated vertex, say $\overline{w}$. Then we claim that $\{v,w\}\in F(G)$.

	Since $\overline{v}$ and $\overline{w}$ are not neighbours in $\overline{G}$, they are neighbours in $G$ by definition, and so to prove this claim it remains only to show that neither $v$ nor $w$ have any other neighbours in $G$.

	Suppose, without loss of generality, that  $v$ had another neighbour, say $a$, in $G$. We must have $\overline{a}\in S$, since $\overline{G}\setminus S=\{\overline{v},\overline{w}\}$. But adding $\overline{a}$ back to $\overline{G}\setminus S$ gives
	
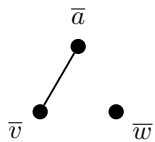
\begin{equation*}
		\begin{tikzpicture}[smallGraph]
			\node[label={-150:$\overline{v}$}](v) at (0,0) {};
      		\node[label={90:$\overline{a}$}](a) at ([shift=(60:1)]v) {};
      	  	\node[label={-30:$\overline{w}$}](w) at ([shift=(0:1)]v) {};

			\foreach \from/\to in {a/w}
				\draw[-] (\from) -- (\to);
		\end{tikzpicture}
	\end{equation*}
	which does not decrease the number of connected components, contradicting that $\overline{a}\in S$. Then $v$ cannot have any other neighbours in $G$, and neither can $w$ by the same argument, so $\{v,w\}\in F(G)$ as claimed.

	If the other connected component of $\overline{G}\setminus S$ has more than one vertex, then there are at least two vertices that are not adjacent to $\overline{v}$ in $\overline{G}$, so $v$ has at least two neighbours in $G$, and therefore $v\in V(H)$.

	We have shown then that either $v\in V(H)$, or $v$ has a single neighbour, say $w$, in $G$, with $\{v,w\}\in F(G)$.

	Note that in the case that $v$ belongs to a free edge $\{v,w\}$, we have
	\begin{equation*}
		N_{\overline{G}}(\overline{v})=V(G)\setminus\{v,w\}
	\end{equation*}

\pagebreak

	Next we will show that $S=N_{\overline{G}}(\overline{v})$. Since $\overline{v}$ is isolated in $\overline{G}\setminus S$, we must have $N_{\overline{G}}(\overline{v})\subseteq S$.

	Now take any $\overline{a}\in S$, so adding $\overline{a}$ back to $\overline{G}\setminus S$ must decrease the number of connected components. We have shown there are only two connected components in $\overline{G}\setminus S$, one of which is $\{\overline{v}\}$, and so $\overline{a}$ must be adjacent to $\overline{v}$ in $\overline{G}$. This means that $S\subseteq N_{\overline{G}}(\overline{v})$, so $S=N_{\overline{G}}(\overline{v})$ as desired.

	We have therefore shown that
	\begin{equation*}
		\mathcal{C}(\overline{G})\subseteq\{\varnothing\}\cup\{N_{\overline{G}}(\overline{v}):v\in V(H)\}\cup\{V(G)\setminus\{v,w\}:\{v,w\}\in F(G)\}
	\end{equation*}
	It remains only to show the reverse inclusion.

	If $\{v,w\}\in F(G)$, then $\overline{v}$ and $\overline{w}$ are adjacent to every vertex in $\overline{G}$ except for each other, so
	\begin{equation*}
		V(G)\setminus\{v,w\}\in\mathcal{C}(\overline{G})
	\end{equation*}
	since adding back any vertex will reconnect $\overline{v}$ and $\overline{w}$.

	Finally, take any $v\in V(H)$, and any $\overline{w}\in N_{\overline{G}}(\overline{v})$. This means that $w$ is not adjacent to $v$ in $G$.

	We know that $v$ has at least two neighbours, say $a_1$ and $a_2$, in $G$, since $v\in V(H)$. If we had the induced subgraph
	
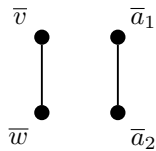
\begin{equation*}
  		\begin{tikzpicture}[smallGraph]
     		\node[label={135:$\overline{v}$}](v) at (0,1) {};
      		\node[label={-135:$\overline{w}$}](w) at (0,0) {};
     		\node[label={[label distance=0.1cm]30:$\overline{a}_1$}](a1) at (1,1) {};
      		\node[label={-45:$\overline{a}_2$}](a2) at (1,0) {};

			\foreach \from/\to in {v/w,a1/a2}
				\draw[-] (\from) -- (\to);
    	\end{tikzpicture}
	\end{equation*}
	in $\overline{G}$ then we would have the $4$-cycle
	
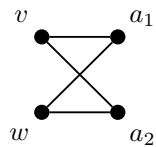
\begin{equation*}
  		\begin{tikzpicture}[smallGraph]
     		\node[label={135:$v$}](v) at (0,1) {};
      		\node[label={-135:$w$}](w) at (0,0) {};
     		\node[label={[label distance=0.1cm]30:$a_1$}](a1) at (1,1) {};
      		\node[label={-45:$a_2$}](a2) at (1,0) {};

			\foreach \from/\to in {v/a1,v/a2,w/a1,w/a2}
				\draw[-] (\from) -- (\to);
    	\end{tikzpicture}
	\end{equation*}
	in $G$, a contradiction. Then $\overline{w}$ must be adjacent to both $\overline{v}$ and one of its neighbours in $\overline{G}$, so reintroducing it will reconnect the two disjoint components of $\overline{G}\setminus S$, and so $N_{\overline{G}}(\overline{v})\in\mathcal{C}(G)$ as desired.
\end{proof}

\begin{note}
	For any distinct $v,w\in V(G)$, we have $N_{\overline{G}}(\overline{v})=N_{\overline{G}}(\overline{w})$ if and only if $\{v,w\}\in F(G)$. To see this, notice that $\overline{v}$ and $\overline{w}$ cannot be adjacent in $\overline{G}$, since they do not belong to their own open neighbourhoods, and so $v$ and $w$ must be adjacent in $G$. We must also have $N_G[v]=N_G[w]$ by the definition of the complement, so neither $v$ nor $w$ can be adjacent to any other vertex in $G$, since such a vertex would then be adjacent to both $v$ and $w$, which would form a $3$-cycle in $G$. In particular, this shows that $N_{\overline{G}}(v)$ is distinct for each $v\in V(H)$.
\end{note}

We will also make use of the following \lcnamecref{ConnectedProp}:

\begin{proposition}\label{ConnectedProp}
	$\overline{G}$ is connected.
\end{proposition}

\begin{proof}
	If we had the induced subgraph
	
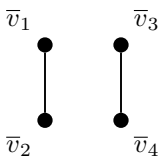
\begin{equation*}
  		\begin{tikzpicture}[smallGraph]
     		\node[label={135:$\overline{v}_1$}](v1) at (0,1) {};
      		\node[label={-135:$\overline{v}_2$}](v2) at (0,0) {};
     		\node[label={45:$\overline{v}_3$}](v3) at (1,1) {};
      		\node[label={-45:$\overline{v}_4$}](v4) at (1,0) {};

			\foreach \from/\to in {v1/v2,v3/v4}
				\draw[-] (\from) -- (\to);
    	\end{tikzpicture}
	\end{equation*}
	in $\overline{G}$, then we would have the $4$-cycle
	
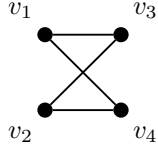
\begin{equation*}
  		\begin{tikzpicture}[smallGraph]
     		\node[label={135:$v_1$}](v1) at (0,1) {};
      		\node[label={-135:$v_2$}](v2) at (0,0) {};
     		\node[label={45:$v_3$}](v3) at (1,1) {};
      		\node[label={-45:$v_4$}](v4) at (1,0) {};

			\foreach \from/\to in {v1/v3,v1/v4,v2/v3,v2/v4}
				\draw[-] (\from) -- (\to);
    	\end{tikzpicture}
	\end{equation*}
	in $G$, a contradiction.

	This means that either $\overline{G}$ is connected, or it contains an isolated vertex $\overline{v}$. But for $\overline{v}$ to be isolated in $\overline{G}$, it must be universal in $G$, contradicting \hyperref[GAssumptions]{our assumptions} on $G$.
\end{proof}

\pagebreak

\subsection{An Equivalence of Posets}

We next define a poset which \hyperref[PosetEquiv]{we will show} is equal to that of \cref{AMPosetDef} for $\overline{G}$.\label{EqualClaim}

\begin{note}
	As  vertex sets, we have
	\begin{equation*}
		N_{\overline{G}}(\overline{v})=V(G)\setminus N_G[v]
	\end{equation*}
	for any $v\in V(H)$.
\end{note}

\begin{definition}
	We first define a set $\mathcal{P}_{\overline{G}}$ consisting of the following ideals:
	\begin{center}
		$\begin{aligned}[t]
			\mathfrak{j}&\defeq\mathcal{J}(K_n)\\
			\mathfrak{a}_v&\defeq(x_i,y_i:i\in V(G)\setminus N_G[v])+\mathcal{J}(K_{N_G(v)})\\
			\mathfrak{a}'_{\{v,w\}}&\defeq(x_i,y_i:i\in V(G)\setminus\{v,w\})\\
			\mathfrak{b}_v&\defeq(x_i,y_i:i\in V(G)\setminus N_G[v])+\mathcal{J}(K_{N_G[v]})\\
			\mathfrak{b}'_{\{v,w\}}&\defeq(x_i,y_i:i\in V(G)\setminus\{v,w\})+(\delta_{v,w})\\
	\mathfrak{c}_{\{v,w\}}&\defeq(x_i,y_i:i\in V(G)\setminus\{v,w\})\\
			\mathfrak{d}_{\{v,w\}}&\defeq(x_i,y_i:i\in V(G)\setminus\{v,w\})+(\delta_{v,w})\\
			\mathfrak{e}_v&\defeq(x_i,y_i:i\in V(G)\setminus\{v\})
		\end{aligned}
		\hspace{0.5cm}
		\begin{aligned}[t]
			&\\
			&\text{\normalfont for $v\in V(H)$}\\
			&\text{\normalfont for $\{v,w\}\in F(G)$}\\
			&\text{\normalfont for $v\in V(H)$}\\
			&\text{\normalfont for $\{v,w\}\in F(G)$}\\
			&\text{\normalfont for $\{v,w\}\in E(H)$}\\
			&\text{\normalfont for $\{v,w\}\in E(H)$}\\
			&\text{\normalfont for $v\in V(H)$ with $\smallAbs{N_H(v)}>1$}
		\end{aligned}$
	\end{center}
	Furthermore, if $\abs{F(G)}\geq2$, or $V(H)\neq\varnothing$ and either $F(G)\neq\varnothing$ or $H$ is not a star, then we also add $\mathfrak{m}$ to this set.

	We next adjoin an element $1_{\mathcal{P}_{\overline{G}}}$, which will be maximal in the poset, and then turn $\mathcal{P}_{\overline{G}}$ into a poset by taking the transitive and reflexive closure of the following relations (the relation on ideals being reverse inclusion):
	\begin{center}
		$\begin{aligned}[t]
    		\mathfrak{m}&\leq\mathfrak{b}'_{\{v,w\}},\mathfrak{e}_v\\
    		\mathfrak{e}_v&\leq\mathfrak{d}_{\{v,w\}}\\
    		\mathfrak{d}_{\{v,w\}}&\leq\mathfrak{b}_v,\mathfrak{b}_w,\mathfrak{c}_{\{v,w\}}\\
    		\mathfrak{c}_{\{v,w\}}&\leq\mathfrak{a}_v,\mathfrak{a}_w\\
    		\mathfrak{b}'_{\{v,w\}}&\leq\mathfrak{j},\mathfrak{a}'_{\{v,w\}}\\
    		\mathfrak{b}_v&\leq\mathfrak{j},\mathfrak{a}_v\\
    		\mathfrak{j},\mathfrak{a}_v,\mathfrak{a}'_{\{v,w\}}&\leq 1_{\mathcal{P}_{\overline{G}}}
\end{aligned}
		\hspace{0.5cm}
		\begin{aligned}[t]
			&\\
			&\text{\normalfont for $w\in N_H(v)$}\\
			&\\
			&\\
			&\\
			&\\
			&
		\end{aligned}$
	\end{center}
\end{definition}

\begin{note}
	By \cref{CompleteGraphPrime}, is easily seen that all of these ideals are prime.
\end{note}

This definition is somewhat complicated, so we will give four examples of increasing complexity:

\begin{example}
	Let
	\begin{equation*}
		G=\quad\begin{tikzpicture}[largeGraph]
			\foreach \v in {1,...,4}
				\node (\v) at (\v,1) {};
			\foreach \i in {1,...,3}
				\pgfmathsetmacro\j{int(\i+1)}
				\draw[line width=0.3mm] (\i.center) -- (\j.center);
			\foreach \v in {1,...,4}
				\fill[black] (\v) circle (2.75pt) node [below=0.1cm] {$\v$};
		\end{tikzpicture}
	\end{equation*}
	so
	\begin{equation*}
		H=\quad\begin{tikzpicture}[largeGraph]
			\foreach \v in {2,...,3}
				\node (\v) at (\v,1) {};
			\foreach \i in {2,...,2}
				\pgfmathsetmacro\j{int(\i+1)}
				\draw[line width=0.3mm] (\i.center) -- (\j.center);
			\foreach \v in {2,...,3}
				\fill[black] (\v) circle (2.75pt) node [below=0.1cm] {$\v$};
		\end{tikzpicture}
	\end{equation*}
	This (as in \cref{P4Example}) yields the poset
	\begin{center}	
		\begin{tikzpicture}[largePosetDiagram]
			\node (1) at (2,6) {$1_{\mathcal{P}_{\overline{G}}}$};
			\node (j) at (1,5) {$\mathfrak{j}$};
			\node (a2) at (2,5) {$\mathfrak{a}_{2}$};
			\node (b2) at (1,4) {$\mathfrak{b}_{2}$};
			\node (a3) at (3,5) {$\mathfrak{a}_{3}$};
			\node (b3) at (2,4) {$\mathfrak{b}_{3}$};
			\node (c2-3) at (3,4) {$\mathfrak{c}_{\{2,3\}}$};
			\node (d2-3) at (2,3) {$\mathfrak{d}_{\{2,3\}}$};

			\foreach \from/\to in {j/1,a2/1,b2/j,b2/a2,a3/1,b3/j,b3/a3,c2-3/a2,c2-3/a3,d2-3/b2,d2-3/b3,d2-3/c2-3}
				\draw[-to] (\from) -- (\to);
		\end{tikzpicture}
	\end{center}
\end{example}

\pagebreak

\begin{example}
	Let
	\begin{equation*}
		G=\quad\begin{tikzpicture}[largeGraph]
			\foreach \v in {1,...,5}
				\node (\v) at (\v,1) {};
			\foreach \i in {1,...,4}
				\pgfmathsetmacro\j{int(\i+1)}
				\draw[line width=0.3mm] (\i.center) -- (\j.center);
			\foreach \v in {1,...,5}
				\fill[black] (\v) circle (2.75pt) node [below=0.1cm] {$\v$};
		\end{tikzpicture}
	\end{equation*}
	so
	\begin{equation*}
		H=\quad\begin{tikzpicture}[largeGraph]
			\foreach \v in {2,...,4}
				\node (\v) at (\v,1) {};
			\foreach \i in {2,...,3}
				\pgfmathsetmacro\j{int(\i+1)}
				\draw[line width=0.3mm] (\i.center) -- (\j.center);
			\foreach \v in {2,...,4}
				\fill[black] (\v) circle (2.75pt) node [below=0.1cm] {$\v$};
		\end{tikzpicture}
	\end{equation*}
	This yields the poset
	\begin{center}
		\begin{tikzpicture}[largePosetDiagram]
			\node (1) at (3,6) {$1_{\mathcal{P}_{\overline{G}}}$};
			\node (j) at (1.5,5) {$\mathfrak{j}$};\node (a2) at (2.5,5) {$\mathfrak{a}_{2}$};
			\node (b2) at (1,4) {$\mathfrak{b}_{2}$};
			\node (a3) at (3.5,5) {$\mathfrak{a}_{3}$};
			\node (b3) at (2,4) {$\mathfrak{b}_{3}$};
			\node (a4) at (4.5,5) {$\mathfrak{a}_{4}$};
			\node (b4) at (3,4) {$\mathfrak{b}_{4}$};
			\node (c2-3) at (4,4) {$\mathfrak{c}_{\{2,3\}}$};
			\node (d2-3) at (2.5,3) {$\mathfrak{d}_{\{2,3\}}$};
			\node (c3-4) at (5,4) {$\mathfrak{c}_{\{3,4\}}$};
			\node (d3-4) at (3.5,3) {$\mathfrak{d}_{\{3,4\}}$};
			\node (e2-3-4) at (3,2) {$\mathfrak{e}_3$};

			\foreach \from/\to in {j/1,a2/1,b2/j,b2/a2,a3/1,b3/j,b3/a3,a4/1,b4/j,b4/a4,c2-3/a2,c2-3/a3,d2-3/b2,d2-3/b3,d2-3/c2-3,c3-4/a3,c3-4/a4,d3-4/b3,d3-4/b4,d3-4/c3-4,e2-3-4/d2-3,e2-3-4/d3-4}
				\draw[-to] (\from) -- (\to);
		\end{tikzpicture}
	\end{center}
\end{example}

\begin{example}
	Let
	\begin{equation*}
		G=\quad\begin{tikzpicture}[largeGraph]
			\foreach \v in {1,...,6}
				\node (\v) at (\v,1) {};
			\node (7) at (3,2) {};
			\node (8) at (3,3) {};
			\foreach \i in {1,...,5}
				\pgfmathsetmacro\j{int(\i+1)}
				\draw[line width=0.3mm] (\i.center) -- (\j.center);
			\draw[line width=0.3mm] (3.center) -- (7.center);
			\draw[line width=0.3mm] (7.center) -- (8.center);
			\foreach \v in {1,...,6}
				\fill[black] (\v) circle (2.75pt) node [below=0.1cm] {$\v$};
			\fill[black] (7) circle (2.75pt) node [left=0.1cm] {$7$};
			\fill[black] (8) circle (2.75pt) node [left=0.1cm] {$8$};
		\end{tikzpicture}
	\end{equation*}
	so
	\begin{equation*}
		H=\quad\begin{tikzpicture}[largeGraph]
			\foreach \v in {2,...,5}
				\node (\v) at (\v,1) {};
			\node (7) at (3,2) {};
			\foreach \i in {2,...,4}
				\pgfmathsetmacro\j{int(\i+1)}
				\draw[line width=0.3mm] (\i.center) -- (\j.center);
			\draw[line width=0.3mm] (3.center) -- (7.center);
			\foreach \v in {2,...,5}
				\fill[black] (\v) circle (2.75pt) node [below=0.1cm] {$\v$};
			\fill[black] (7) circle (2.75pt) node [left=0.1cm] {$7$};
		\end{tikzpicture}
	\end{equation*}

	This yields the poset
	\begin{center}
		\begin{tikzpicture}[largePosetDiagram]
			\node (1) at (5,6) {$1_{\mathcal{P}_{\overline{G}}}$};
			\node (j) at (2.5,5) {$\mathfrak{j}$};
			\node (a2) at (3.5,5) {$\mathfrak{a}_{2}$};
			\node (b2) at (1,4) {$\mathfrak{b}_{2}$};
			\node (a3) at (4.5,5) {$\mathfrak{a}_{3}$};
			\node (b3) at (2,4) {$\mathfrak{b}_{3}$};
			\node (a4) at (5.5,5) {$\mathfrak{a}_{4}$};
			\node (b4) at (3,4) {$\mathfrak{b}_{4}$};
			\node (a5) at (6.5,5) {$\mathfrak{a}_{5}$};
			\node (b5) at (4,4) {$\mathfrak{b}_{5}$};
			\node (a7) at (7.5,5) {$\mathfrak{a}_{7}$};
			\node (b7) at (5,4) {$\mathfrak{b}_{7}$};
			\node (c2-3) at (6,4) {$\mathfrak{c}_{\{2,3\}}$};
			\node (d2-3) at (3.5,3) {$\mathfrak{d}_{\{2,3\}}$};
			\node (c3-4) at (7,4) {$\mathfrak{c}_{\{3,4\}}$};
			\node (d3-4) at (4.5,3) {$\mathfrak{d}_{\{3,4\}}$};
			\node (c3-7) at (8,4) {$\mathfrak{c}_{\{3,7\}}$};
			\node (d3-7) at (5.5,3) {$\mathfrak{d}_{\{3,7\}}$};
			\node (c4-5) at (9,4) {$\mathfrak{c}_{\{4,5\}}$};
			\node (d4-5) at (6.5,3) {$\mathfrak{d}_{\{4,5\}}$};
			\node (e2-3-4-7) at (4.5,2) {$\mathfrak{e}_3$};
			\node (e3-4-5) at (5.5,2) {$\mathfrak{e}_4$};
			\node (m) at (5,1) {$\mathfrak{m}$};

			\foreach \from/\to in {j/1,a2/1,b2/j,b2/a2,a3/1,b3/j,b3/a3,a4/1,b4/j,b4/a4,a5/1,b5/j,b5/a5,a7/1,b7/j,b7/a7,c2-3/a2,c2-3/a3,d2-3/b2,d2-3/b3,d2-3/c2-3,c3-4/a3,c3-4/a4,d3-4/b3,d3-4/b4,d3-4/c3-4,c3-7/a3,c3-7/a7,d3-7/b3,d3-7/b7,d3-7/c3-7,c4-5/a4,c4-5/a5,d4-5/b4,d4-5/b5,d4-5/c4-5,e2-3-4-7/d2-3,e2-3-4-7/d3-4,e2-3-4-7/d3-7,m/e2-3-4-7,e3-4-5/d3-4,e3-4-5/d4-5,m/e3-4-5}
				\draw[-to] (\from) -- (\to);
		\end{tikzpicture}
	\end{center}
\end{example}

\begin{example}
	Let
	\begin{equation*}
		G=\quad\begin{tikzpicture}[largeGraph]
			\foreach \v in {1,...,6}
				\node (\v) at (\v,1) {};
			\node (7) at (3,2) {};
			\node (8) at (3,3) {};
			\node (9) at (7,1) {};
			\node (10) at (7,2) {};
			\foreach \i in {1,...,5}
				\pgfmathsetmacro\j{int(\i+1)}
				\draw[line width=0.3mm] (\i.center) -- (\j.center);
			\draw[line width=0.3mm] (3.center) -- (7.center);
			\draw[line width=0.3mm] (7.center) -- (8.center);
			\draw[line width=0.3mm] (9.center) -- (10.center);
			\foreach \v in {1,...,6}
				\fill[black] (\v) circle (2.75pt) node [below=0.1cm] {$\v$};
			\fill[black] (7) circle (2.75pt) node [left=0.1cm] {$7$};
			\fill[black] (8) circle (2.75pt) node [left=0.1cm] {$8$};
			\fill[black] (9) circle (2.75pt) node [right=0.1cm] {$9$};
			\fill[black] (10) circle (2.75pt) node [right=0.1cm] {$10$};
		\end{tikzpicture}
	\end{equation*}
	so again
	\begin{equation*}
		H=\quad\begin{tikzpicture}[largeGraph]
			\foreach \v in {2,...,5}
				\node (\v) at (\v,1) {};
			\node (7) at (3,2) {};
			\foreach \i in {2,...,4}
				\pgfmathsetmacro\j{int(\i+1)}
				\draw[line width=0.3mm] (\i.center) -- (\j.center);
			\draw[line width=0.3mm] (3.center) -- (7.center);
			\foreach \v in {2,...,5}
				\fill[black] (\v) circle (2.75pt) node [below=0.1cm] {$\v$};
			\fill[black] (7) circle (2.75pt) node [left=0.1cm] {$7$};
		\end{tikzpicture}
	\end{equation*}

	This yields the poset
	\begin{center}
		\begin{tikzpicture}[largePosetDiagram]
			\node (1) at (5,6) {$1_{\mathcal{P}_{\overline{G}}}$};
			\node (a9-10) at (1,5) {$\mathfrak{a}'_{\{9,10\}}$};
			\node (j) at (2.5,5) {$\mathfrak{j}$};
			\node (b9-10) at (0,4) {$\mathfrak{b}'_{\{9,10\}}$};
			\node (a2) at (3.5,5) {$\mathfrak{a}_{2}$};
			\node (b2) at (1,4) {$\mathfrak{b}_{2}$};
			\node (a3) at (4.5,5) {$\mathfrak{a}_{3}$};
			\node (b3) at (2,4) {$\mathfrak{b}_{3}$};
			\node (a4) at (5.5,5) {$\mathfrak{a}_{4}$};
			\node (b4) at (3,4) {$\mathfrak{b}_{4}$};
			\node (a5) at (6.5,5) {$\mathfrak{a}_{5}$};
			\node (b5) at (4,4) {$\mathfrak{b}_{5}$};
			\node (a7) at (7.5,5) {$\mathfrak{a}_{7}$};
			\node (b7) at (5,4) {$\mathfrak{b}_{7}$};
			\node (c2-3) at (6,4) {$\mathfrak{c}_{\{2,3\}}$};
			\node (d2-3) at (3.5,3) {$\mathfrak{d}_{\{2,3\}}$};
			\node (c3-4) at (7,4) {$\mathfrak{c}_{\{3,4\}}$};
			\node (d3-4) at (4.5,3) {$\mathfrak{d}_{\{3,4\}}$};
			\node (c3-7) at (8,4) {$\mathfrak{c}_{\{3,7\}}$};
			\node (d3-7) at (5.5,3) {$\mathfrak{d}_{\{3,7\}}$};
			\node (c4-5) at (9,4) {$\mathfrak{c}_{\{4,5\}}$};
			\node (d4-5) at (6.5,3) {$\mathfrak{d}_{\{4,5\}}$};
			\node (e2-3-4-7) at (4.5,2) {$\mathfrak{e}_3$};
			\node (e3-4-5) at (5.5,2) {$\mathfrak{e}_4$};
			\node (m) at (5,1) {$\mathfrak{m}$};

			\foreach \from/\to in {j/1,a2/1,b2/j,b2/a2,a3/1,b3/j,b3/a3,a4/1,b4/j,b4/a4,a5/1,b5/j,b5/a5,a7/1,b7/j,b7/a7,c2-3/a2,c2-3/a3,d2-3/b2,d2-3/b3,d2-3/c2-3,c3-4/a3,c3-4/a4,d3-4/b3,d3-4/b4,d3-4/c3-4,c3-7/a3,c3-7/a7,d3-7/b3,d3-7/b7,d3-7/c3-7,c4-5/a4,c4-5/a5,d4-5/b4,d4-5/b5,d4-5/c4-5,e2-3-4-7/d2-3,e2-3-4-7/d3-4,e2-3-4-7/d3-7,m/e2-3-4-7,e3-4-5/d3-4,e3-4-5/d4-5,m/e3-4-5,a9-10/1,b9-10/j,b9-10/a9-10,m/b9-10}
				\draw[-to] (\from) -- (\to);
		\end{tikzpicture}
	\end{center}
\end{example}

Having given these examples, we will now show the \hyperref[EqualClaim]{claimed} equality:

\begin{lemma}\label{PosetEquiv}
	We have $\mathcal{Q}_{\mathcal{J}(\overline{G})}=\mathcal{P}_{\overline{G}}$ as posets.
\end{lemma}

\begin{proof}
	We will first show that $\mathcal{Q}_{\mathcal{J}(\overline{G})}\subseteq\mathcal{P}_{\overline{G}}$.

	By \cref{ComplementCutSets}, $\mathfrak{j}$, the $\mathfrak{a}_v$, and the $\mathfrak{a}'_{\{v,w\}}$ are precisely the associated primes of $\mathcal{J}(\overline{G})$.

	We first address the $\mathfrak{a}'_{\{v,w\}}$. We have:
	\begin{enumerate}
		\item $\mathfrak{a}'_{\{v,w\}}+\mathfrak{j}=\mathfrak{b}'_{\{v,w\}}$.
		\item $\mathfrak{a}'_{\{v_1,w_1\}}+\mathfrak{a}'_{\{v_2,w_2\}}=\mathfrak{m}$ for any distinct $\{v_1,w_1\}$ and $\{v_2,w_2\}$ in $F(G)$.
		\item $\mathfrak{a}'_{\{v,w\}}+\mathfrak{a}_u=\mathfrak{m}$ for any $u\in V(H)$.
	\end{enumerate}
	If the second situation arises we must have $\abs{F(G)}\geq2$, so $\mathfrak{m}\in\mathcal{P}_{\overline{G}}$. If the third situation arises we must have $V(H)\neq\varnothing$ and $F(G)\neq\varnothing$, so again $\mathfrak{m}\in\mathcal{P}_{\overline{G}}$.

	Now take any non-empty $S\subseteq V(H)$, and set
	\begin{equation*}
		\sigma=\sum_{v\in S}\mathfrak{a}_{v}
	\end{equation*}
	If we have some $v_1$ and $v_2$ in $S$ which are not adjacent and share no neighbours, then
	\begin{equation*}
		N_G[v_1]\cap N_G[v_2]=\varnothing
	\end{equation*}
	and so $\sigma=\mathfrak{m}$. Note that, if this is the case, then $H$ cannot be a star, since no vertex in $H$ is adjacent to both $v_1$ and $v_2$, so $\mathfrak{m}\in\mathcal{P}_{\overline{G}}$.

	We then assume that each vertex in $S$ is either adjacent to, or shares a neighbour with, every other.

\pagebreak

	If $S=\{v_1,v_2\}$, then $v_1$ and $v_2$ must be adjacent, so $\sigma=\mathfrak{c}_{\{v_1,v_2\}}$, since they cannot have a common neighbour as this would mean that $G$ contains a $3$-cycle.

	Then assume that $S$ consists of at least $3$ vertices. At least $2$ of these vertices, say $v_1$ and $v_2$, cannot be adjacent, since otherwise $G$ would contain a $3$-cycle. They must also share a common neighbour, say $w$, or we would be in the case $\sigma=\mathfrak{m}$. Note that $w\in V(H)$, since we know it has at least two neighbours. If $v_1$ and $v_2$ had another common neighbour then $G$ would contain a $4$-cycle, so this cannot be the case, and therefore
	\begin{equation*}
		N_G[v_1]\cap N_G[v_2]=\{w\}
	\end{equation*}
	We then have
	\begin{equation*}
		\mathfrak{a}_{v_1}+\mathfrak{a}_{v_2}=\mathfrak{e}_w
	\end{equation*}
	Since
	\begin{equation*}
		\mathfrak{e}_w+\mathfrak{a}_v=\begin{cases}
			\mathfrak{e}_w & \text{if $w\in N_G[v]$}\\
			\mathfrak{m} & \text{otherwise}
		\end{cases}
	\end{equation*}
	we now know that $\sigma$ must be either $\mathfrak{a}_v$, $\mathfrak{c}_{\{v_1,v_2\}}$, $\mathfrak{e}_w$ or $\mathfrak{m}$ for some $v,v_1,v_2,w\in V(H)$.

	Furthermore, we have
	\begin{align*}
		\mathfrak{a}_v+\mathfrak{j}&=\mathfrak{b}_v\\
		\mathfrak{c}_{\{v,w\}}+\mathfrak{j}&=\mathfrak{d}_{\{v,w\}}\\
		\mathfrak{e}_w+\mathfrak{j}&=\mathfrak{e}_w\\
		\mathfrak{m}+\mathfrak{j}&=\mathfrak{m}
	\end{align*}
	and so $\mathcal{Q}_{\mathcal{J}(\overline{G})}$ is contained in $\mathcal{P}_{\overline{G}}$.

	We will now show that $\mathcal{P}_{\overline{G}}\subseteq\mathcal{Q}_{\mathcal{J}(\overline{G})}$. This containment follows almost by definition:

	Since $\mathfrak{j}$, the $\mathfrak{a}_v$, and the $\mathfrak{a}'_{\{v,w\}}$ are precisely the associated primes of $\mathcal{J}(\overline{G})$ by \cref{ComplementCutSets}, they trivially belong to $\mathcal{Q}_{\mathcal{J}(\overline{G})}$.

	Furthermore, we have
	\begin{center}
		$\begin{aligned}[t]
			\mathfrak{b}_v&=\mathfrak{a}_v+\mathfrak{j}\\
			\mathfrak{b}'_{\{v,w\}}&=\mathfrak{a}'_{\{v,w\}}+\mathfrak{j}\\
			\mathfrak{c}_{\{v,w\}}&=\mathfrak{a}_v+\mathfrak{a}_w\\
			\mathfrak{d}_{\{v,w\}}&=\mathfrak{c}_{\{v,w\}}+\mathfrak{j}\\
			\mathfrak{e}_v&=\mathfrak{d}_{\{v,w_1\}}+\mathfrak{d}_{\{v,w_2\}}
		\end{aligned}
		\hspace{0.5cm}
		\begin{aligned}[t]
			&\text{\normalfont for $v\in V(H)$}\\
			&\text{\normalfont for $\{v,w\}\in F(G)$}\\
			&\text{\normalfont for $\{v,w\}\in E(H)$}\\
			&\text{\normalfont for $\{v,w\}\in E(H)$}\\
			&\text{\normalfont for $v\in V(H)$ with $w_1,w_2\in N_H(v)$}
		\end{aligned}$
	\end{center}
	If $\mathfrak{m}$ belongs to $\mathcal{P}_{\overline{G}}$, then, by definition, at least one of the following must hold:
	\begin{enumerate}
		\item There exist some distinct $\{v_1,w_1\}$ and $\{v_2,w_2\}$ in $F(G)$.
		\item There exists some $u\in V(H)\neq\varnothing$ and $\{v,w\}\in F(G)$.
		\item There exist $v,w\in V(H)$ which are not adjacent in $H$.
	\end{enumerate}
	In each case:
	\begin{enumerate}
		\item $\mathfrak{m}=\mathfrak{a}'_{\{v_1,w_1\}}+\mathfrak{a}'_{\{v_2,w_2\}}$.
		\item $\mathfrak{m}=\mathfrak{a}_u+\mathfrak{a}'_{\{v,w\}}$.
		\item $\mathfrak{m}=\mathfrak{a}_v+\mathfrak{a}_w$.
	\end{enumerate}
	so $\mathfrak{m}$ belongs to $\mathcal{Q}_{\mathcal{J}(\overline{G})}$, and the desired containment is established.
	
	Clearly the orders on $\mathcal{Q}_{\mathcal{J}(\overline{G})}$ and $\mathcal{P}_{\overline{G}}$ agree, and so $\mathcal{Q}_{\mathcal{J}(\overline{G})}=\mathcal{P}_{\overline{G}}$ as posets as claimed.
\end{proof}

\begin{note}\label{PPosetNote}
	$\mathcal{P}_{\overline{G}}$ is also equal to the poset $\mathcal{P}_{\hspace{-0.05cm}\mathcal{J}(\overline{G})}$ of \cite[Section 3.1]{alvarezmontanerLocalCohomologyBinomial2020}, since each sum of the associated primes of $\mathcal{J}(\overline{G})$ is itself prime.
\end{note}

\pagebreak

It is straightforward to calculate the dimensions of the ideals in $\mathcal{P}_{\overline{G}}$:
\begin{proposition}\label{DimProp}
	We have
	\begin{align*}
		\dim(R/\mathfrak{j})&=n+1\\
		\dim(R/\mathfrak{a}_v)&=\smallAbs{N_G[v]}+2\\
		\dim(R/\mathfrak{a}'_{\{v,w\}})&=4\\
		\dim(R/\mathfrak{b}_v)&=\smallAbs{N_G[v]}+1\\
		\dim(R/\mathfrak{b}'_{\{v,w\}})&=3\\
		\dim(R/\mathfrak{c}_{\{v,w\}})&=4\\
		\dim(R/\mathfrak{d}_{\{v,w\}})&=3\\
		\dim(R/\mathfrak{e}_v)&=2
	\end{align*}
\end{proposition}

\begin{proof}
	This follows immediately from \cref{CompleteGraphDim}.
\end{proof}

\pagebreak

\subsection{The Main Theorem \& Some Corollaries}

With these results in hand, we can now completely describe the local cohomology modules of $R/\mathcal{J}(\overline{G})$:

\begin{theorem}\label{MainGCThm}
	We have the following isomorphisms
	{\allowdisplaybreaks\begin{align*}
		&H_\mathfrak{m}^4(R/\mathcal{J}(\overline{G}))\cong\hspace{-0.1cm}\left[\hspace{0.6cm}\bigoplus_{\mathclap{\{v,w\}\in F(G)}}\,[H_\mathfrak{m}^4(R/\mathfrak{a}'_{\{v,w\}})\oplus H_\mathfrak{m}^3(R/\mathfrak{b}'_{\{v,w\}})]\right]\\
		&H_\mathfrak{m}^5(R/\mathcal{J}(\overline{G}))\cong\hspace{-0.1cm}\left[\hspace{0.4cm}\bigoplus_{\mathclap{\substack{v\in V(H)\\\smallAbs{N_G[v]}=3}}}\,[H_\mathfrak{m}^5(R/\mathfrak{a}_v)\oplus H_\mathfrak{m}^4(R/\mathfrak{b}_v)]\right]\hspace{-0.1cm}\oplus\hspace{-0.1cm}\left[\hspace{0.6cm}\bigoplus_{\mathclap{\{v,w\}\in E(H)}}\,[H_\mathfrak{m}^4(R/\mathfrak{c}_{\{v,w\}})\oplus H_\mathfrak{m}^3(R/\mathfrak{d}_{\{v,w\}})]\right]\\
		&H_\mathfrak{m}^i(R/\mathcal{J}(\overline{G}))\cong\hspace{-0.1cm}\left[\hspace{0.55cm}\bigoplus_{\mathclap{\substack{v\in V(H)\\\smallAbs{N_G[v]}=i-2}}}\,[H_\mathfrak{m}^i(R/\mathfrak{a}_v)\oplus H_\mathfrak{m}^{i-1}(R/\mathfrak{b}_v)]\right]\\
		&H_\mathfrak{m}^{n+1}(R/\mathcal{J}(\overline{G}))\cong H_\mathfrak{m}^{n+1}(R/\mathfrak{j})\oplus\hspace{-0.0625cm}\left[\hspace{0.6cm}\bigoplus_{\mathclap{\substack{v\in V(H)\\\smallAbs{N_G[v]}=n-1}}}\,[H_\mathfrak{m}^{n+1}(R/\mathfrak{a}_v)\oplus H_\mathfrak{m}^n(R/\mathfrak{b}_v)]\right]
	\end{align*}}
	of graded $k$-vector spaces (for $5<i<n+1$), with all other local cohomology modules vanishing.
\end{theorem}

\begin{proof}
	We aim to apply \cref{AMMainTheorem} to $\mathcal{Q}_{\mathcal{J}(\overline{G})}$.
	
	Note that, since every vertex in $H$ has at least two neighbours in $G$, we cannot have $\smallAbs{N_G[v]}<3$ for any $v\in V(H)$, and since $G$ has no universal vertex \hyperref[GAssumptions]{by assumption}, we cannot have $\smallAbs{N_G[v]}>n-1$.

	Taking into account \cref{DimProp}, the following values are easily computed:
	\begin{center}
		\begin{tblr}{colspec={c|c|c|c|c|c},hline{2} = {1}{-}{solid},hline{2} = {2}{-}{solid},vline{2-6} = {abovepos = 1, belowpos = 1},stretch=1.75}
			$\mathfrak{q}$ & $\dim_R(R/\mathfrak{q})$ & $(\mathfrak{q},1_{\mathcal{Q}_{\mathcal{J}(\overline{G})}})$ & $\dim_k(\widetilde{H}^{-1})$ & $\dim_k(\widetilde{H}^0)$ & $\dim_k(\widetilde{H}^1)$ \\
			$\mathfrak{j}$ & $n+1$ & $\varnothing$ & $1$ & $0$ & $0$ \\
			\hline
			$\mathfrak{a}_v$ & $\smallAbs{N_G[v]}+2$ & $\varnothing$ & $1$ & $0$ & $0$ \\
			\hline
			$\mathfrak{a}'_{\{v,w\}}$ & $4$ & $\varnothing$ & $1$ & $0$ & $0$ \\
			\hline
			$\mathfrak{b}_v$ & $\smallAbs{N_G[v]}+1$ &
			\begin{tikzpicture}[smallPosetDiagram]
				\node (j) at (0,1.035) {$\mathfrak{j}$};
				\node (av) at (1,1) {$\mathfrak{a}_v$};
			\end{tikzpicture}
			& $0$ & $1$ & $0$ \\
			\hline
			$\mathfrak{b}'_{\{v,w\}}$ & $3$ &
			\begin{tikzpicture}[smallPosetDiagram]
				\node (j) at (0,1.035) {$\mathfrak{j}$};
				\node (avw) at (1,1) {$\mathclap{\hspace{0.925cm}\mathfrak{a}'_{\{v,w\}}}\hphantom{\mathfrak{a}_v}$};
			\end{tikzpicture}
			& $0$ & $1$ & $0$ \\
			\hline
			$\mathfrak{c}_{\{v,w\}}$ & $4$ &
			\begin{tikzpicture}[smallPosetDiagram]
				\node (av) at (0,1) {$\mathfrak{a}_v$};
				\node (aw) at (1,1) {$\mathfrak{a}_w$};
			\end{tikzpicture}
			& $0$ & $1$ & $0$ \\
			\hline
			$\mathfrak{d}_{\{v,w\}}$ & $3$ &
			\begin{tikzpicture}[smallPosetDiagram]
				\node (j) at (0,1) {$\mathfrak{j}$};
				\node (av) at (1,1) {$\mathfrak{a}_v$};
				\node (aw) at (2,1) {$\mathfrak{a}_w$};
				\node (bv) at (0,0) {$\mathfrak{b}_v$};
				\node (bw) at (1,0) {$\mathfrak{b}_w$};
				\node (cv-w) at (2,0) {$\mathfrak{c}_{\{v,w\}}$};

				\foreach \from/\to in {bv/j,bv/av,bw/j,bw/aw,cv-w/av,cv-w/aw}
					\draw[-to] (\from) -- (\to);
			\end{tikzpicture}
			& $0$ & $0$ & $1$
		\end{tblr}
	\end{center}
	As \hyperref[PPosetNote]{noted} earlier, our poset $\mathcal{P}_{\overline{G}}=\mathcal{Q}_{\mathcal{J}(\overline{G})}$ coincides with the poset $\mathcal{P}_{\mathcal{J}(\overline{G})}$ of \cite[Section 3.1]{alvarezmontanerLocalCohomologyBinomial2020}. It then also coincides with the poset defined in \cite[Definition 3.1]{rouzbahanimalayeriBinomialEdgeIdeals2021}, and so we may apply \cite[Corollary 4.6]{rouzbahanimalayeriBinomialEdgeIdeals2021} to deduce that all reduced cohomology of $(\mathfrak{m},1_{\mathcal{Q}_{\mathcal{J}(\overline{G})}})$ and each $(\mathfrak{e}_v,1_{\mathcal{Q}_{\mathcal{J}(\overline{G})}})$ vanishes, so $\mathfrak{m}$ and each $\mathfrak{e}_v$ contribute no summands to the local cohomology modules. This concludes the proof.
\end{proof}

\pagebreak

We can use this theorem to calculate the dimension, depth, and regularity of $R/\mathcal{J}(\overline{G})$. In the case where $k$ is of prime characteristic $p>0$, we can also calculate the cohomological dimension and bound the arithmetic rank:

\begin{corollary}\label{MainGCCorr}
	We have
	\begin{align*}
		\depth_R(R/\mathcal{J}(\overline{G}))&=\begin{cases}
			4 & \text{\normalfont if $F(G)\neq\varnothing$}\\
			5 & \text{\normalfont otherwise}
		\end{cases}\\
		\dim(R/\mathcal{J}(\overline{G}))&=n+1
	\end{align*}
	and $R/\mathcal{J}(\overline{G})$ is Cohen-Macaulay if and only if $G$ is either $P_4$ or $P_3\sqcup K_1$.
\end{corollary}

\begin{proof}
	If $F(G)\neq\varnothing$ then $H_\mathfrak{m}^4(R/\mathcal{J}(\overline{G}))$ contains a non-vanishing summand, and so will be the bottom local cohomology module. Otherwise, there must exist a vertex of $G$ with degree at least $2$ (since $G\neq nK_1$ by \hyperref[GAssumptions]{assumption}), so $V(H)\neq\varnothing$ and therefore $H_\mathfrak{m}^5(R/\mathcal{J}(\overline{G}))$ contains a non-vanishing summand.

	Furthermore, $H_\mathfrak{m}^{n+1}(R/\mathcal{J}(\overline{G}))$ always contains the non-vanishing summand $H_\mathfrak{m}^{n+1}(R/\mathfrak{j})$, and so will always be the top local cohomology module.

	This tells us that $R/\mathcal{J}(\overline{G})$ is Cohen-Macaulay if and only if $n+1$ equals $4$ in the case that $F(G)\neq\varnothing$, or $5$ otherwise. \hyperref[GAssumptions]{We assumed} that $n>3$, so we must have $n=4$, and is it easily checked that the only such graphs satisfying $F(G)=\varnothing$ and our \hyperref[GAssumptions]{other criteria} are $P_4$ and $P_3\sqcup K_1$, so we are done.
\end{proof}

\begin{note}
	If $F(G)\neq\varnothing$, we can write $G=K_2\sqcup G'$ for some subgraph $G'$ of $G$, and so $\overline{G}=(2K_1)\ast\overline{G'}$ (where $\ast$ denotes the graph join). Then, when $n\geq4$, \cref{MainGCCorr} immediately demonstrates a special case of \cite[Theorem 5.3]{rouzbahanimalayeriBinomialEdgeIdeals2021}, which says that the binomial edge ideal of a graph on at least $4$ vertices has depth $4$ if and only if it is the join of $2K_1$ and some other graph.
\end{note}

\begin{corollary}\label{GCPrimeCharCorr}
	Suppose that $k$ has prime characteristic $p>0$, and set
	\begin{equation*}
		d\defeq\begin{cases}
			4 & \text{\normalfont if $F(G)\neq\varnothing$}\\
			5 & \text{\normalfont otherwise}
		\end{cases}
	\end{equation*}
	Then
	\begin{equation*}
		\cd_R(\mathcal{J}(\overline{G}))=2n-d
	\end{equation*}
	and
	\begin{equation*}
		2n-d\leq\ara_R(\mathcal{J}(\overline{G}))\leq2n
	\end{equation*}
\end{corollary}

\begin{proof}
	Note that
	\begin{equation*}
		\cd_R(\mathcal{J}(\overline{G}))=\max\{1\leq i\leq 2n:H_{\mathcal{J}(\overline{G})}^i(R)\neq0\}
	\end{equation*}
	by \cite[Theorem 9.6]{iyengarTwentyFourHoursLocal2007}.

	We will make use of the functor $\prescript{\ast\hspace{-0.1cm}}{}{\mathscr{H}}$ defined in \cite[Theorem 4.2]{lyubeznikLocalCohomologyModules2016}. This is the graded counterpart to the functor $\mathscr{H}$ defined in \cite[Theorem 4.2]{lyubeznikFmodulesApplicationsLocal1997}.

	By \cite[Proposition 2.8]{lyubeznikLocalCohomologyModules2016}, we have that
	\begin{equation*}
		\prescript{\ast\hspace{-0.1cm}}{}{\mathscr{H}}(H_\mathfrak{m}^i(R/\mathcal{J}(\overline{G})))\cong H_{\mathcal{J}(\overline{G})}^{2n-i}(R)
	\end{equation*}
	Now, Frobenius is injective on $H_\mathfrak{m}^i(R/\mathcal{J}(\overline{G}))$ by \cite[Corollary 3.6]{destefaniFrobeniusMethodsCombinatorics2022}, and so the first claim follows by \cite[Theorem 2.5~(2)]{lyubeznikLocalCohomologyModules2016} and \cref{MainGCThm}.

	We have
	\begin{equation*}
		\ara_R(\mathcal{J}(\overline{G}))\leq2n
	\end{equation*}
	by \cite[Theorem 9.13 \& Remark 9.14]{iyengarTwentyFourHoursLocal2007}, and
	\begin{equation*}
		2n-d\leq\ara_R(\mathcal{J}(\overline{G}))
	\end{equation*}
	by \cite[Proposition 9.12]{iyengarTwentyFourHoursLocal2007}, since we have shown that the cohomological dimension is $2n-d$.
\end{proof}

\begin{note}
	At the time of writing, we do not know if \cref{GCPrimeCharCorr} holds in characteristic $0$.
\end{note}

\pagebreak

\subsubsection*{Computing the Regularity}

In order to compute $\reg_R(R/\mathcal{J}(\overline{G}))$, we will first calculate $\reg_R(R/\mathfrak{q})$ for each $\mathfrak{q}\in\mathcal{Q}_{\mathcal{J}(\overline{G})}$ appearing in the local cohomology modules of \cref{MainGCThm}, then show that we can use these calculations to determine $\reg_R(R/\mathcal{J}(\overline{G}))$ itself.

We will need a few results about regularity before we can compute the $\reg_R(R/\mathfrak{q})$, these are not difficult, but we include proofs for completeness:

\begin{proposition}\label{ExtraVarReg}
	Let $R$ be a polynomial ring over a field with the standard grading, $I$ a homogeneous ideal of $R$, and $z$ an indeterminate. Then
	\begin{equation*}
		\reg_{R[z]}(R[z]/(IR[z]+(z)))=\reg_R(R/I)
	\end{equation*}
\end{proposition}

\begin{proof}
	Note that
	\begin{equation*}
		R[z]/(IR[z]+(z))\cong R/I
	\end{equation*}
	as $R[z]$-modules if we let $z$ act on $R/I$ by killing any element. Then the result follows immediately from applying \cite[Corollary 4.6]{eisenbudGeometrySyzygiesSecond2005} to $R/I$ and $R\hookrightarrow R[z]$.
\end{proof}

\begin{proposition}\label{FreeVarReg}
	Let $R$ be a polynomial ring over a field with the standard grading, $I$ a homogeneous ideal of $R$, and $z$ an indeterminate. Then
	\begin{equation*}
		\reg_{R[z]}(R[z]/IR[z])=\reg_R(R/I)
	\end{equation*}
\end{proposition}

\begin{proof}
	A minimal graded free resolution of $R/I$ over $R$ remains so over $R[z]$, since $R[z]$ is free (and therefore flat) over $R$, and its grading is compatible with that of $R$.
\end{proof}

We will make use of the following notation:

\begin{notation}
	Let $1\leq i\leq n$, and let $G$ be a graph with vertices $1,\ldots,j$ for some $j\leq i$. Then we set
		\begin{equation*}
		R'_i\defeq k[x_1,\ldots,x_i,y_1,\ldots,y_i]
	\end{equation*}
	and denote by $\mathcal{J}_i(G)$ the binomial edge ideal of $G$ in $R'_i$.
\end{notation}

\begin{proposition}\label{CompleteGraphRegLemma}
	For any $m\geq2$, we have
	\begin{equation*}
		\reg_{R'_m}\hspace{-0.05cm}(R'_m/\mathcal{J}_m\hspace{-0.05cm}(K_m))=1
	\end{equation*}
\end{proposition}

\begin{proof}
	Since
	\begin{equation*}
		\reg_{R'_m}\hspace{-0.05cm}(\mathcal{J}_m(K_m))=2
	\end{equation*}
	(see for example \cite[Remark 3.3]{kianiBinomialEdgeIdeals2012}), the result is immediate from the fact that
	\begin{equation*}
		\reg_{R'_m}\hspace{-0.05cm}(R/I)=\reg_{R'_m}\hspace{-0.05cm}(I)-1
	\end{equation*}
	for any (homogeneous) ideal $I$ of $R'_m$.
\end{proof}

\pagebreak

We can now compute the necessary $\reg_R(R/\mathfrak{q})$:

\begin{proposition}\label{RegProp}
	We have
	\begin{equation*}
		\reg_R(R/\mathfrak{q})=\begin{cases}
			1 & \text{\normalfont for $\mathfrak{q}\in\{\mathfrak{j},\mathfrak{a}_v,\mathfrak{b}_v,\mathfrak{d}_{\{v,w\}},\mathfrak{b}'_{\{v,w\}}\}$}\\
			0 & \text{\normalfont for $\mathfrak{q}\in\{\mathfrak{c}_{\{v,w\}},\mathfrak{a}'_{\{v,w\}}\}$}
		\end{cases}
	\end{equation*}
\end{proposition}

\begin{proof}\mbox{}
	\newlength{\mylength}\maxlength{\ref{jReg},\ref{aReg},\ref{bReg},\ref{cReg},\ref{dReg},\ref{aPrimeReg},\ref{bPrimeReg}}{\mylength}
	\begin{enumerate}[labelindent=0pt,labelwidth=\mylength,leftmargin=!]
		\myitem{$\mathfrak{j}$:}\label{jReg} We immediately have $\reg_R(R/\mathfrak{j})=1$ by \cref{CompleteGraphRegLemma}.
	\end{enumerate}

	Suppose that we have $v\in V(H)$, and set $m=\smallAbs{N_G(v)}$. We may relabel the vertices of $G$ so that $N_G(v)=\{1,\ldots,m\}$ and $v=m+1$. Then:
	\begin{enumerate}[labelindent=0pt,labelwidth=\mylength,leftmargin=!]
		\myitem{$\mathfrak{a}_v$:}\label{aReg} Note that
		\begin{equation*}
			R/\mathfrak{a}_v\cong R'_{m+1}/\mathcal{J}_{m+1}(K_m)\cong(R'_m/\mathcal{J}_m(K_m))[x_{m+1},y_{m+1}]
		\end{equation*}
		We have
		\begin{equation*}
			\reg_{R'_m}\hspace{-0.05cm}(R'_m/\mathcal{J}_m(K_m))=1
		\end{equation*}
		by \cref{CompleteGraphRegLemma}. Next, we have
		\begin{equation*}
			\reg_{R'_{m+1}}\hspace{-0.05cm}(R'_{m+1}/\mathcal{J}_{m+1}(K_m))=\reg_{R'_{m+1}}\hspace{-0.05cm}((R'_m/\mathcal{J}_m(K_m))[x_{m+1},y_{m+1}])=1
		\end{equation*}
		with the second equality following by \cref{FreeVarReg}, and so applying \cref{ExtraVarReg} we obtain $\reg_R(R/\mathfrak{a}_v)=1$.
		\myitem{$\mathfrak{b}_v$:}\label{bReg} In this case, we have
		\begin{equation*}
			R/\mathfrak{b}_v\cong R'_{m+1}/\mathcal{J}_{m+1}(K_{m+1})
		\end{equation*}
		and
		\begin{equation*}
			\reg_{R'_{m+1}}\hspace{-0.05cm}(R'_{m+1}/\mathcal{J}_{m+1}(K_{m+1}))=1
		\end{equation*}
		by \cref{CompleteGraphRegLemma}, and so, again by applying \cref{ExtraVarReg}, we have $\reg_R(R/\mathfrak{b}_v)=1$.
	\end{enumerate}
	Now suppose that we have some $\{v,w\}\in E(H)$, and relabel the vertices of $G$ again so that $v=1$ and $w=2$. Then:
	\begin{enumerate}[labelindent=0pt,labelwidth=\mylength,leftmargin=!]
		\myitem{$\mathfrak{c}_{\{v,w\}}$:}\label{cReg} We have $R/\mathfrak{c}_{\{v,w\}}\cong R'_2$, and $R'_2$ clearly has regularity $0$ over itself, so by \cref{ExtraVarReg} we have $\reg_R(R/\mathfrak{c}_{\{v,w\}})=0$.
		\myitem{$\mathfrak{d}_{\{v,w\}}$:}\label{dReg} We have
		\begin{equation*}
			R/\mathfrak{d}_{\{v,w\}}\cong R'_2/\mathcal{J}_2(K_2)
		\end{equation*}
		This has regularity $1$ over $R'_2$ by \cref{CompleteGraphRegLemma}, and so, as before by \cref{ExtraVarReg}, we have $\reg_R(R/\mathfrak{d}_{\{v,w\}})=1$.
	\end{enumerate}
	Finally, suppose that we have some $\{v',w'\}\in F(G)$. Then:
	\begin{enumerate}[labelindent=0pt,labelwidth=\mylength,leftmargin=!]
		\myitem{$\mathfrak{a}'_{\{v',w'\}}$:}\label{aPrimeReg} Since $R/\mathfrak{a}'_{\{v',w'\}}\cong R/\mathfrak{c}_{\{v,w\}}$, we have $\reg_R(R/\mathfrak{a}'_{\{v',w'\}})=0$.
		\myitem{$\mathfrak{b}'_{\{v',w'\}}$:}\label{bPrimeReg} Since $R/\mathfrak{b}'_{\{v',w'\}}\cong R/\mathfrak{d}_{\{v,w\}}$, we have $\reg_R(R/\mathfrak{b}'_{\{v',w'\}})=1$.\qedhere
	\end{enumerate}
\end{proof}

We will make use of the following characterisation of regularity, beginning with a definition:

\begin{definition}
	Let $R$ be a polynomial ring over a field with the standard grading, and $M$ a $\mathbb{Z}$-graded $R$-module. Let $M_d$ denote the $d$\textsuperscript{th} graded component of $M$. Then we set
	\begin{equation*}
		\lastGrade_R(M)\defeq\max\{d\in\mathbb{Z}:M_d\neq0\}
	\end{equation*}
	with $\lastGrade_R(M)=-\infty$ if $M=0$, or $\lastGrade_R(M)=\infty$ if no such $d$ exists.
\end{definition}

\begin{lemma}\label{RegLCChar}
	For any finitely generated graded $R$-module $M$, we have
	\begin{equation*}
		\reg_R(M)=\max\{\lastGrade_R(H_\mathfrak{m}^i(M))+i:i\geq 0\}
	\end{equation*}
\end{lemma}

\begin{proof}
	See, for example, \cite[Corollary 4.5]{eisenbudGeometrySyzygiesSecond2005}.
\end{proof}

\pagebreak

To calculate $R/\mathcal{J}(\overline{G})$, we will prove a general result:

\begin{theorem}\label{RegCalcThm}
	Let $G$ be \emph{any} graph. Then, with notation as in \cref{AMMainTheorem}, we have
	\begin{equation*}
		\reg_R(R/\mathcal{J}(G))=\max_{\mathclap{\substack{\vspace{0.025cm}\\\mathfrak{q}\in\mathcal{Q}_{\mathcal{J}(G)}\\\vspace{-0.075cm}\\r\geq0}}}\{\reg_R(R/\mathfrak{q})-d_\mathfrak{q}+r:M_{r,\mathfrak{q}}\neq0\}
	\end{equation*}
\end{theorem}

\begin{proof}
	By \cref{AMMainTheorem}, we have
	\begin{equation*}
		H_\mathfrak{m}^r(R/\mathcal{J}(G))\cong\bigoplus_{\mathclap{\mathfrak{q}\in\mathcal{Q}_{\mathcal{J}(G)}}}H_\mathfrak{m}^{d_\mathfrak{q}}(R/\mathfrak{q})^{M_{r,\mathfrak{q}}}
	\end{equation*}
	for all $r\geq0$.

	Since each $\mathfrak{q}\in\mathcal{Q}_{\mathcal{J}(G)}$ is Cohen-Macaulay, $R/\mathfrak{q}$ has only a single non-vanishing local cohomology module $H_\mathfrak{m}^{d_\mathfrak{q}}(R/\mathfrak{q})$, and so, by \cref{RegLCChar}, we have
	\begin{equation*}
		\lastGrade_R(H_\mathfrak{m}^{d_\mathfrak{q}}(R/\mathfrak{q}))=\reg_R(R/\mathfrak{q})-d_\mathfrak{q}
	\end{equation*}
	Then
	\begin{equation*}
		\lastGrade_R(H_\mathfrak{m}^r(R/\mathcal{J}(G)))=\max_{\mathclap{\substack{\vspace{0.025cm}\\\mathfrak{q}\in\mathcal{Q}_{\mathcal{J}(G)}}}}\{\lastGrade_R(H_\mathfrak{m}^{d_\mathfrak{q}}(R/\mathfrak{q})):M_{r,\mathfrak{q}}\neq0\}=\max_{\mathclap{\substack{\vspace{0.025cm}\\\mathfrak{q}\in\mathcal{Q}_{\mathcal{J}(G)}}}}\{\reg_R(R/\mathfrak{q})-d_\mathfrak{q}:M_{r,\mathfrak{q}}\neq0\}
	\end{equation*}
	so, again by \cref{RegLCChar}, we have
	\begin{align*}
		\reg_R(R/\mathcal{J}(G))&=\max\{\lastGrade_R(H_\mathfrak{m}^r(R/\mathcal{J}(G)))+r:r\geq 0\}\\
		&=\max_{\mathclap{\substack{\vspace{0.025cm}\\\mathfrak{q}\in\mathcal{Q}_{\mathcal{J}(G)}\\\vspace{-0.075cm}\\r\geq0}}}\{\reg_R(R/\mathfrak{q})-d_\mathfrak{q}+r:M_{r,\mathfrak{q}}\neq0\}
	\end{align*}
	as desired.
\end{proof}

\begin{corollary}
	We have
	\begin{equation*}
		\reg_R(R/\mathcal{J}(\overline{G}))=\begin{cases}
			2 & \text{\normalfont if $E(H)=\varnothing$}\\
			3 & \text{\normalfont otherwise}
		\end{cases}
	\end{equation*}
\end{corollary}

\begin{proof}
	By \cref{DimProp}, \cref{RegProp}, and \cref{MainGCThm}, we have
	\begin{center}
		\begin{tblr}{colspec={c|c|c|c},hline{2} = {1}{-}{solid},hline{2} = {2}{-}{solid},vline{2-4} = {abovepos = 1, belowpos = 1},stretch=1.25}
			$r$ & $\mathfrak{q}$  & $\reg_R(R/\mathfrak{q})-d_\mathfrak{q}+r$ & $M_{r,\mathfrak{q}}\neq0$ \\
			\SetCell[r=2]{c} $4$ & $\mathfrak{a}'_{\{v,w\}}$ & $0$ & \SetCell[r=2]{c} If $F(G)\neq\varnothing$ \\
			\hline
			& $\mathfrak{b}'_{\{v,w\}}$ & $2$ & \\
			\hline
			\SetCell[r=2]{c} $5$ & $\mathfrak{c}_{\{v,w\}}$ & $1$ & \SetCell[r=2]{c} If $E(H)\neq\varnothing$ \\
			\hline
			& $\mathfrak{d}_{\{v,w\}}$ & $3$ & \\
			\hline
			\SetCell[r=2]{c} $i$ & $\mathfrak{a}_v$ & $1$ & \SetCell[r=2]{c} {For $i\in\{\smallAbs{N_G[v]}+2:v\in V(H)\}$} \\
			\hline
			& $\mathfrak{b}_v$ & $2$ & \\
			\hline
			$n+1$ & $\mathfrak{j}$ & $1$ & Always
		\end{tblr}
	\end{center}
	We noted in the proof of \cref{MainGCCorr} that at least one of $F(G)$ and $V(H)$ must be non-empty, and so the result follows from \cref{RegCalcThm}.
\end{proof}

\newpage
\printbibliography

\end{document}